\documentclass[11pt]{article}
\usepackage{amsmath,amsthm,amssymb,amscd,color}

\newcommand{\bbR}{\mathbb{R}}

\newcommand{\norm}[1]{\left\lVert #1 \right\rVert}

\newcommand{\bbQ}{\mathbb{Q}}
\newcommand{\less}{\mathord{<}}
\newcommand{\restrict}{\upharpoonright}
\newcommand{\HC}{\mathsf{HC}}
\newcommand{\NHC}{\mathsf{NHC}}

\newtheorem{theorem}{Theorem}[section]
\newtheorem{lemma}[theorem]{Lemma}

\newtheorem{proposition}[theorem]{Proposition}

\theoremstyle{definition}
\newtheorem{definition}[theorem]{Definition}

\newtheorem{remark}[theorem]{Remark}
\newtheorem*{MA}{Martin's Axiom}

\theoremstyle{remark}
\newtheorem*{claim*}{Claim}

\author{Konstantinos A. Beros
\\
{\small Miami University}
\and
Paul B. Larson
\thanks{We thank Rebecca Sanders for introducing us to the problem that originally motivated this paper, and Sophie Grivaux and Etienne Matheron for comments on earlier drafts of this paper.  The research of the second author is supported in part by NSF grants DMS-1201494 and DMS-1764320.}
\\
{\small Miami University}
}
\title{Common hypercyclic vectors for unilateral weighted shifts on $\ell^2$}
\date{}

\begin{document}

\maketitle

\begin{abstract}
Each $w \in \ell^\infty$ defines a bounded linear operator $B_w : \ell^2 \rightarrow \ell^2$ where $B_w (x) (i) = w(i) \cdot x(i+1)$ for each $i \in \omega$. A vector $x \in \ell^2$ is {\em hypercyclic} for $B_w$ if the set 
$\{ B_w^k (x) : k \in \omega\}$ of forward iterates of $x$ is dense in $\ell^2$.  For each such $w$, the set $\HC (w)$ consisting of all vectors hypercyclic for $B_w$ is $G_{\delta}$. The set of \emph{common hypercyclic vectors} for a set $W \subseteq \ell^\infty$ is the set $\HC^* (W) = \bigcap_{w \in W} \HC (w)$. We show that $\HC^* (W)$ can be made arbitrarily complicated by making $W$ sufficiently complex, and that even for a $G_\delta$ set $W$ the set $\HC^* (W)$ can be non-Borel.

Finally, by assuming the Continuum Hypothesis or Martin's Axiom, we are able to construct a set $W$ such that $\HC^* (W)$ does not even have the property of Baire.
\end{abstract}

\section{Introduction}

Given a separable Banach space $X$ and a linear operator $T : X \rightarrow X$, one says that an element $x \in X$ is a {\em hypercyclic vector} for $T$ iff the forward iterates of $x$ under $T$ form a dense subset of $X$.  That is, $x$ is hypercyclic for $T$ iff the set
\[
D = \{ T^n (x) : n \in \omega\}
\]
is dense in $X$.  In the context of the present paper, $X$ will always be the Hilbert space $\ell^2$.  It is worth noting that if $X$ has no isolated points, then $x \in X$ is hypercyclic for $T$ iff the set $D$ above has {\em infinite} intersection with every open subset of $X$.

The study of hypercyclic vectors is quite old, dating back to work of Birkhoff from the 1920s and McClane from the 1950s.  Somewhat more recently, hypercyclicity was studied in the context of {\em scaled unilateral shifts} on $\ell^2$, i.e., operators $T_\lambda : \ell^2 \rightarrow \ell^2$ of the form
\[
T_\lambda (x) (i) = \lambda x(i+1)
\]
for a fixed $\lambda \in \mathbb C$.  In \cite{rolewicz}, Rolewicz showed that the maps $T_\lambda$ all posess hypercyclic vectors provided $|\lambda| > 1$.  In general, the hypercyclic vectors of $T_\lambda$ (for $|\lambda| > 1$) form a dense $G_\delta$ subset of $\ell^2$.  It thus follows from the Baire category theorem that any countable family of such $T_\lambda$ posess common hypercyclic vectors.  Expanding upon this observation, Abakumov and Gordon \cite{abakumov-gordon} showed in 2003 that the family 
\[
\{ T_\lambda : |\lambda|>1 \}
\]
has common hypercyclic vectors.  Notice that this result does not follow from the Baire category theorem as the family of operators in question is uncountable.  They remark that they do not know of other examples of uncountable families of operators with common hypercyclic vectors.  

We take up a similar study in the setting of unilateral weighted shifts.  Given a bounded sequence $w \in \ell^\infty$, the associated {\em unilateral weighted shift} is the operator $B_w : \ell^2 \rightarrow \ell^2$ defined by
\[
B_w (x) (i) = w(i) x(i+1)
\]
The sequence $w \in \ell^\infty$ is referred to as a {\em weight sequence}.  Note that a scaled shift as described above is simply a weighted shift with constant weight sequence.  Let $\HC (w)$ denote the set of hypercyclic vectors for $B_w$ and, given a set $W \subseteq \ell^\infty$, let
\[
\HC^* (W) = \bigcap_{w \in W} \HC (w)
\]
be the set of hypercyclic vectors common to all $B_w$ for $w \in W$.  It is routine to check that $\HC (w)$ is a $G_\delta$ set for any $w \in \ell^\infty$.  If $W$ is countable, then $\HC^* (W)$ will also be $G_\delta$.  Moreover, if $W \subseteq \ell^\infty$ is countable and each $\HC (w)$ is nonempty (for $w \in W$), then (by Lemma~\ref{dense if nonempty}) $\HC^* (W)$ is also nonempty.

Things become more interesting when $W$ is uncountable.

In the first place, $\HC^* (W)$ may be empty even if each $\HC (w) \neq \emptyset$ (for $w \in W$).  For instance, it is not hard to show (Lemma~\ref{infly many twos}) that if $w$ is a sequence of ones and twos with infinitely many twos, then $\HC (w)$ is always nonempty.  On the other hand, if $W$ is the collection of all such 1,2-sequences, then $\HC^* (W)$ is empty.  Indeed, given any $y \in \ell^2$, there is a $w \in W$ such that $y \notin \HC (w)$: simply make the twos of $w$ sufficiently sparse that 
\[
\prod_{i < n} w(i) = 2^k
\]
only if $\lVert y \upharpoonright [n,\infty)\rVert_2 < 2^{-k}$.  This will guarantee that, for all but finitely many $n$, the norm $\lVert B_w^n (y) \rVert_2$ is less than $1$ and hence $y \notin \HC (w)$.  

Moreover, whereas $\HC^* (W)$ is $G_\delta$ if $W$ is countable, by allowing $W$ to be uncountable, $\HC^* (W)$ may be made arbitrarily complex, e.g., not $G_\delta$ or non-Borel (Theorem~\ref{T1}).  Exactly what ``complex'' means in this context will be explained below.  Although arbitrary complexity cannot be achieved when $W$ is Borel, the set $\HC^*(W)$ can still be made non-Borel even for $G_\delta$ sets $W$.  Taking this one step further, by assuming the Continuum Hypothesis (or Martin's Axiom), a set $W$ can constructed so that $\HC^* (W)$ fails to have the Baire property (Theorem ~\ref{non baire under ch}).

A by-product of these results is the construction of a large class of examples of uncountable sets $W$ for which $\HC^* (W)$ is nonempty.  This addresses Abakumov and Gordon's comment about the lack of examples of uncountable families of operators with common hypercyclic vectors.

Before proceeding to the precise statements of our results, it is necessary to introduce a few preliminaries and some terminology from descriptive set theory.  

As is common in descriptive set theory, the setting is that of separable completely metrizable spaces (i.e., {\em Polish spaces}), e.g., $\mathbb R$, $\mathbb C$ or $\ell^2$.  More typically, however, one studies spaces which are countable powers (equipped with the product topology), i.e., sequence spaces of the form
\[
A^\omega = \{ \big( a(i) \big)_{i = 0}^\infty : (\forall i) (a(i) \in A)\}
\]
where $A$ is some fixed Polish space (often a finite discrete space).  In the present setting two spaces will be of particular interest: $\mathbb R^\omega$ (the space of sequences of real numbers) and $\{ 1 , 2 \}^\omega$ (the spaces of 1,2-sequences).  The former will give us an overlying structure for the spaces $\ell^\infty$ and $\ell^2$, while the latter will be the source of our weight sequences $w$ as we study the phenomenon of hypercyclicity.

Next we must make precise what we mean by ``complexity'' in a topological space.  In this context, the key descriptive set theoretic concept is that of a {\em pointclass}.  There are many variations on the definition of the term ``pointclass''.  For the purposes of the present work, we will use the following definition.

\begin{definition}
A {\em pointclass} $\Gamma$ is a collection of subsets of Polish (separable completely metrizable) spaces such that
\begin{itemize}
	\item $\Gamma$ is closed under continuous preimages,
	\item $\Gamma$ is closed under finite unions and
	\item $\Gamma$ is closed under finite intersections.
\end{itemize}
Given a pointclass $\Gamma$, the {\em dual pointclass} $\bar \Gamma$ consists of those $Y$ contained in some Polish space $X$ such that $X \setminus Y \in \Gamma$.  A pointclass is {\em non-self-dual} iff there exist a Polish space $X$ and a set $Y \subseteq X$ such that $Y \in \Gamma$ but $Y \notin \bar \Gamma$ (equivalently, $X \setminus Y \notin \Gamma$).
\end{definition}

To take a few examples, ``closed'' and ``open'' are dual pointclasses as are ``$F_\sigma$'' and ``$G_\delta$''.  All four of these classes are non-self-dual.  Pointclasses are used as a measure of complexity in set theory.  For instance, open and closed sets are regarded as ``simple''.  Sets which are strictly $F_\sigma$, $G_\delta$, $G_{\delta \sigma}$, etc., are regarded as more complicated.  In general, the number of iterated intersections and unions required to build a set from the open and closed sets is the quantifiable measure of complexity in descriptive set theory.

This sort of analysis only applies to the Borel sets.  In moving beyond the Borel sets, the natural next step is the analytic and co-analytic sets.  A set is {\em analytic} if it is the continuous image of a Borel set.  A set is {\em co-analytic} if its complement is analytic.  (Classically, analytic sets were known as {\em Souslin sets}.)  It is a standard result in descriptive set theory (Souslin's theorem) that a set which is both analytic and co-analytic is actually Borel.  (See 14.11 in Kechris \cite{kechris dst}.)

An equivalent (and more useful) formulation says that a subset $S$ of a Polish space $X$ is analytic iff there exists a Polish space $Y$ and a Borel set $B \subseteq X \times Y$ such that
\[
S = \{ x \in X : (\exists y ) ((x,y) \in B)\}.
\]
Likewise, a set $P \subseteq X$ is co-analytic iff it has the form
\[
P = \{ x \in X : (\forall y) ((x,y) \in B\}
\]
where $B \subseteq X \times Y$ is again Borel.  See section 14A of Kechris \cite{kechris dst} for an introduction to analytic sets and the relationship between the various ways of describing them.

In light of this, we may now establish an upper bound on the complexity of $\HC^* (W)$ when $W$ is analytic.

\begin{proposition}\label{hc* is co-analytic}
If $W \subseteq \ell^\infty$ is analytic in the product topology on $\mathbb R^\omega$, the common hypercyclic vectors $\HC^* (W)$ form a co-analytic set.
\end{proposition}

\begin{proof}
To see this, observe that, for $y \in \ell^2$,
\[
y \in \bigcap_{w \in W} \HC (w) \iff (\forall w \in \ell^\infty) (w \in W \implies y \in \HC (w)).
\]
Thus,  $\bigcap_{w \in W} \HC (w)$ is co-analytic since the relation
\[
P (y,w) \iff y \in \HC (w)
\]
is itself $G_\delta$ and the $\mathbf \Sigma^1_1$ relation ``$w \in W$'' is in the hypothesis of the conditional statement above.
\end{proof}

It is worth remarking that the Borel structure of $\ell^\infty$ (as a Banach space) is fundamentally different from the inherited structure of $\mathbb R^\omega$.  In fact, since the cellularity of $\ell^\infty$ is $\mathfrak c$ (the cardinality of the real numbers), it is possible to construct sets which are closed in $\ell^\infty$, but not even Borel in $\mathbb R^\omega$.  Therefore, when considering the topological complexity of sets of weight sequences in $\ell^\infty$, we will use the Borel structure of $\mathbb R^\omega$.  In fact, in all of our constructions, the weight sequences used are in $W \subseteq \{ 1 , 2 \}^\omega$ (a compact Polish space).


Our main results show in a very strong way that the upper bound from the last proposition cannot be improved.  Moreover, they establish that the sets $\HC^* (W)$ can be made arbitrarily complex (in the descriptive set theoretic) sense.  This answers a question posed to us by Rebecca Sanders and which served as the original motivation for this paper: are there sets of weight sequences $W$ for which $\HC^* (W)$ is non-$G_\delta$?

\begin{theorem}\label{T1}
Given a non-self-dual pointclass $\Gamma$ which contains the closed sets, there is a set $W \subseteq \{1,2\}^\omega$ such that $\bigcap_{w \in W} \HC (w)$ is not in $\Gamma$.
\end{theorem}

\begin{theorem}\label{T2}
There is a Borel set $W \subseteq  \{1,2\}^\omega$ such that $\bigcap_{w \in W} \HC (w)$ is properly co-analytic, i.e., not analytic.  In particular, $\HC^* (W)$ is not Borel.
\end{theorem}

Theorem~\ref{T2} is more concrete than Theorem~\ref{T1} in the sense that the complexity of $W$ is highly restricted by the requirement that it be $G_\delta$.  

The third main theorem of this paper uses Martin's Axiom (MA) to construct a set $W \subseteq \{ 1 , 2 \}^\omega$ such that $\HC^* (W)$ does not have the property of Baire.

\begin{theorem}\label{non baire under ch}
Assuming MA, there exists $W \subseteq \{ 1 , 2 \}^\omega$ such that $\HC^* (W)$ does not have the property of Baire.
\end{theorem}

The proof of Theorem~\ref{non baire under ch} is essentially a more complex version of the construction under the continuum hypothesis (CH) of a Bernstein set.  Also note that, since MA is a consequence of CH, Theorem~\ref{non baire under ch} is a consequence of CH as well.  In fact, the proof under CH is somewhat simpler than the Martin's Axiom version.

For the reader unfamiliar with Martin's Axiom, it is (at an intuitive level) a strengthening of the Baire Category Theorem to encompass arbitrary intersections of fewer than continuum many dense open sets.  A detailed statement of Martin's Axiom is included in the next section.

\section{Fundamentals}

\subsection{Notation}

Let $\lVert \cdot \rVert_2$ denote the usual $\ell^2$ norm.  In what follows, this notation will be used for finite sequences as well, i.e., for $s \in \mathbb R^{<\omega}$,
\[
\lVert s \rVert_2 = \sqrt{s(0)^2 + \ldots + s(n)^2}
\]
assuming $s$ is of length $n+1$.

The notation $|s|$ will be used to denote both the length of a string (if $s \in 2^{< \omega}$) and the length of an interval (if $s \subseteq \omega$ is an interval).  The notation $\lVert x \rVert_\infty$ will denote the $\ell^\infty$- or sup-norm of $x$.  Again, this definition makes sense for any string $x$ -- either finite or infinite.

Since the topology of $\ell^2$ is a strict refinement of the product topology on $\mathbb R^\omega$, a couple minor lemmas are required to permit the use of some ``product topology intuition'' when working in $\ell^2$.  The first of these lemmas indicates a relationship between the 2-norm and the sup-norm of a finite string which will be quite useful. 

\begin{lemma}\label{LSQ}
If $s$ is a finite string of real numbers, having length $n$,
\[
\lVert s \rVert_2 \leq n^{1/2} \lVert s \rVert_\infty.
\]

\end{lemma}

\begin{proof}
Suppose that $s \in \mathbb R^n$ and $\lVert s \rVert_\infty = n^{-1/2} \cdot \varepsilon$ for some positive $\varepsilon$.  In other words, $|s(i)| \leq n^{-1/2} \varepsilon$ for all $i < n$.  It follows that
\begin{align*}
\lVert s \rVert_2
&= \sqrt{s(0)^2 + \ldots + s(n-1)^2}\\
&\leq \sqrt{n \cdot (n^{-1/2} \varepsilon)^2}\\
&= \varepsilon\\
&= n^{1/2} \lVert s \rVert_\infty
\end{align*}
This proves the lemma.
\end{proof}

\subsection{Topology in $\ell^2$}

In much of this paper, it will be helpful to use an alternative topological basis for $\ell^2$.  Given a finite nonempty string $q \in \mathbb Q^{<\omega}$ of rationals and a (rational) number $\varepsilon > 0$, let
\[
U_{q , \varepsilon} = \{ x \in \ell^2 : \lVert (x \upharpoonright |q|) - q \rVert_\infty < \varepsilon |q|^{-1/2}  \mbox{ and } \lVert x \upharpoonright [ |q| , \infty ) \rVert_2 < \varepsilon\}.
\]
In the case that $q = \langle \ \rangle$ is the empty sequence, simply let 
\[
U_{q , \varepsilon} = \{ x \in \ell^2 : \lVert x \rVert_2 < \varepsilon\}.
\]
First note that each $U_{q , \varepsilon}$ is open.  In order to check that the $U_{q , \varepsilon}$ form a basis for $\ell^2$, fix a basic open ball
\[
V = \{ x \in \ell^2 : \lVert x - x_0 \rVert_2 < \varepsilon\}
\]
where $x_0 \in \ell^2$ and $\varepsilon  >0$ are fixed.  Let $n \in \omega$ be such that
\[
\lVert x_0 \upharpoonright [n, \infty) \rVert_2 < \varepsilon /4
\]
and choose $q \in \mathbb Q^n$ such that
\[
\lVert x_0 \upharpoonright n - q \rVert_\infty < (\varepsilon / 4) \cdot n^{-1/2}.
\]
First of all, it follows from the definition of $U_{q , \varepsilon}$ that $x_0 \in U_{q , \varepsilon/4}$.  To see that $U_{q , \varepsilon/4} \subseteq V$, observe that if $x \in U_{q , \varepsilon/4}$,
\begin{align*}
\lVert x - x_0 \rVert_2
&\leq \lVert (x - x_0) \upharpoonright n\rVert_2  + \lVert (x - x_0) \upharpoonright [n , \infty)\rVert_2\\
&\leq n^{1/2} \lVert (x - x_0) \upharpoonright n\rVert_\infty + \lVert x \upharpoonright [n , \infty) \rVert_2 + \lVert x_0 \upharpoonright [n , \infty) \rVert_2\\
&< n^{1/2} (\lVert (x \upharpoonright n) - q\rVert_\infty + \lVert (x_0 \upharpoonright n) - q\rVert_\infty) + \varepsilon/4 + \varepsilon / 4\\
&< n^{1/2} ((\varepsilon / 4) n^{-1/2}  + (\varepsilon / 4) n^{-1/2} ) + \varepsilon/2\\
&= \varepsilon
\end{align*}
As $x \in U_{q , \varepsilon / 4}$ was arbitrary, it follows that $U_{q , \varepsilon / 4} \subseteq V$.  Since $V$ was an arbitrary open ball, this shows that the $U_{q , \varepsilon}$ form a topological basis for $\ell^2$.

The next lemma relates the continuity of functions into $\ell^2$ with respect to two different topologies: the subspace topology inherited from $\mathbb R^\omega$ and the inherent Banach space topology induced by the 2-norm.

\begin{lemma}\label{L0}
If $A$ is a countable set and $f : 2^A \rightarrow \ell^2$ is such that
\begin{enumerate}
\item $f$ is continuous with respect to the product topology on $2^A$ and the subspace topology on $\ell^{2}$ inherited from $\mathbb{R}^{\omega}$, and

\item there exists $y \in \ell^2$ such that $|f(x)(i)| \leq y(i)$ for all $x \in 2^A$ and $i \in \omega$,
\end{enumerate}
then $f$ is continuous with respect to the norm-topology on $\ell^2$.
\end{lemma}

\begin{proof}
Let $y \in \ell^2$ be as in the statement of the lemma.  Towards the goal of showing that $f$ is $\ell^2$-continuous, fix $\varepsilon > 0$ and let $n$ be such that
\[
\lVert y \upharpoonright [n, \infty) \rVert_2 < \varepsilon/4.
\]
Since $f$ is continuous into the subspace topology on $\ell^2$ (from $\mathbb R^\omega$), and $2^A$ is compact, there exists a finite $F \subseteq A$ such that, for $x_1 , x_2 \in 2^A$, if $x_1 \upharpoonright F = x_2 \upharpoonright F$, then
\[
|f(x_1) (i) - f(x_2) (i)| < n^{-1/2} \varepsilon/2
\]
for all $i < n$.  In particular, $x_1 \upharpoonright F = x_2 \upharpoonright F$ guarantees
\[
\lVert ( ( f(x_1) - f(x_2) ) \upharpoonright n ) \rVert_2 < \varepsilon/2
\]
by Lemma \ref{LSQ}.  It now follows that, whenever $x_1 , x_2 \in 2^A$ and $x_1 \upharpoonright F = x_2 \upharpoonright F$,
\begin{align*}
\lVert f(x_1) - f(x_2) \rVert_2
&\leq \lVert ( ( f(x_1) - f(x_2) ) \upharpoonright n ) \rVert_2 + \lVert f(x_1) \upharpoonright [n , \infty) \rVert_2\\
& \qquad \qquad + \lVert f(x_2) \upharpoonright [n , \infty) \rVert_2\\
&< \varepsilon/2 + 2 \lVert y \upharpoonright [n , \infty) \rVert_2\\
&< \varepsilon/2 + 2 \varepsilon/4 = \varepsilon.
\end{align*}
Since $\varepsilon$ was arbitrary this completes the proof.  Note that since $2^A$ is compact, $f$ is in fact uniformly continuous with respect to the standard ultrametric on $2^A$.
\end{proof}

\subsection{Hypercyclic vectors}

The next two lemmas are standard.  See, e.g., Theorems 1.2 and 1.40 in \cite{BayartMatheron}.

\begin{lemma}\label{dense if nonempty}
If $W \subseteq \{ 1 , 2 \}^\omega$ and $\HC^* (W)$ is nonempty, then $\HC^* (W)$ is dense.
\end{lemma}

\begin{proof}
Suppose that $y \in \HC^* (W)$.  Fix an open set $U \subseteq \ell^2$ and let $s \in R^{< \omega}$ be such that $y + s {}^\smallfrown \bar 0  \in U$.  To see that $y + s {}^\smallfrown \bar 0 \in \HC^* (W)$, fix $w \in W$ and an open set $V \subseteq \ell^2$.  Since $\ell^2$ is $T_1$, there are infinitely many $k \in \omega$ such that $B^k_w (y) \in V$.  In particular, there exists $k \geq |s|$ with this property.  For such a $k \geq |s|$,
\[
B^k_w (y + s {}^\smallfrown \bar 0) = B^k_w (y) \in V.
\]
As $w$ and $V$ were arbitrary, it follows that $y + s {}^\smallfrown \bar 0 \in \HC^* (W)$.  Since $U$ was arbitrary as well, it follows that $\HC^* (W)$ is dense in $\ell^2$.
\end{proof}

\begin{lemma}\label{infly many twos}
If $w \in \{ 1 , 2 \}^\omega$ is such that that there are infinitely many $i$ with $w(i) = 2$, then $\HC (w)$ is comeager.
\end{lemma}

\begin{proof}
First of all, to see that $\HC (w)$ is nonempty, let $U_{q_n , \varepsilon_n}$ enumerate the basic open neighborhoods defined above.  Given a sequence $k_0 < k_1 < \ldots$, define for each $n$, a string $\bar q_n \in \mathbb Q^{<\omega}$ by 
\[
\bar q_n (i) = 2^{-| \{ j \in [i , i + k_n) : w (j) = 2\} |} \cdot q_n (i)
\]
(for $i < |q_n|$).  Note that if $y \in \ell^2$ has a copy of $\bar q_n$ starting at the $k_n$th bit of $y$, then $B^{k_n}_w (y)$ begins with a copy of $q_n$.  All that remains is to choose a specific sequence $k_0 < k_1 < \ldots$ which grows quickly enough that, for each $n$, 
\begin{itemize}
\item $k_{n+1} - k_n \geq |\bar q_n|$ and
\item $\lVert \bar q_n \rVert_2 \leq 2^{-n-1-k_n} \cdot \min \{ \varepsilon_j : j \leq n \}$.
\end{itemize}
Having done this, let $p_n = k_{n+1} - k_n - |\bar q_n|$ and 
\[
y = \bar q_0 {}^\smallfrown 0^{p_0} {}^\smallfrown \bar q_1 {}^\smallfrown 0^{p_1} {}^\smallfrown \ldots
\]
Then for each $n$, $y$ has a copy of $\bar q_n$ beginning at its $k_n$th term and so
\[
B^{k_n}_w (y) = q_n {}^\smallfrown \alpha
\]
for some $\alpha \in \ell^2$.  Moreover, by the choice of $k_{n+1} < k_n < \ldots$, 
\begin{align*}
\lVert \alpha \rVert_2 
&\leq \sum_{i > n} 2^{k_n} \cdot \lVert \bar q_i \rVert_2\\
&\leq \sum_{i > n} 2^{k_n} \cdot 2^{-i-1-k_i} \cdot \min \{ \varepsilon_j : j \leq n \}\\
&\leq \sum_{i > n} 2^{-i-1} \varepsilon_n\\
&= 2^{-n-1} \varepsilon_n < \varepsilon_n
\end{align*}
It follows that $B^k_w (y) \in U_{q_n , \varepsilon_n}$.  As $n$ was arbitrary, it follows that $y \in \HC (w)$.

It now follows from Lemma~\ref{dense if nonempty} (applied to $W = \{ w \}$) that $\HC (w)$ is dense.  As a dense $G_\delta$ set, $\HC (w)$ is thus comeager.
\end{proof}

\subsection{Martin's Axiom}

Given a partially ordered set $( \mathbb P , < )$, a subset $D \subseteq \mathbb P$ is {\em dense} iff, for each $p \in \mathbb P$, there is a $q \leq p$ with $q \in D$.  A set $G \subseteq \mathbb P$ is called a {\em filter} iff 
\begin{itemize}
\item $(\forall p , q \in \mathbb P) (p \geq q \in G \implies p \in G)$ and
\item $(\forall p , q \in G) (\exists r \in G) (r \leq p \mbox{ and } r \leq q)$.
\end{itemize}
Given a cardinal number $\kappa < \mathfrak c$ (the cardinality of $\mathbb R$), Martin's Axiom (MA) is the following assertion.

\begin{MA}
Given the following:
\begin{itemize}
\item a partially ordered set $\mathbb P$ with no uncountable antichains, i.e., $\mathbb P$ has the {\em countable chain condition} (or {\em ccc}), and
\item a collection $\mathcal D \subseteq \mathcal P (\mathbb P)$ of dense subsets of $\mathbb P$ with $| \mathcal D | < \mathfrak c$,
\end{itemize}
there exists a filter $G \subseteq \mathbb P$ such that $G \cap D \neq \emptyset$ for all $D \in \mathcal D$.
\end{MA}

Though consistent with the standard axioms of set theory, Martin's Axiom (like CH) is not a consequence of them.  A well-known consequence of MA asserts that the intersection of fewer than $\mathfrak c$-many comeager sets is still comeager.  See, for instance, Theorem 16.23 in Jech \cite{jech}.  In this sense, MA may be regarded as a strengthening of the Baire category theorem.

\section{Proof of Theorem~\ref{T1}}

Example 7.1 of Bayart-Matheron \cite{BayartMatheron} shows that there is no $y \in \ell^2$ which is hypercyclic for every $w \in \{ 1 , 2 \}^\omega$.  The argument in this section is motivated by a set theoretic construction which uses the proof of this fact.  We briefly sketch the idea:

For each $\epsilon > 0$, the map sending each $y \in \ell^{2}$ to the least $n$ such that $\norm{y \restrict [n, \infty)}_{2} < \epsilon$ is
Borel (but not continuous). It follows from the argument of Example 7.1 of \cite{BayartMatheron} that there is a Borel function $b$ sending each $y$ to a hypercyclic $w \in \{1,2\}$ for which $y$ is not hypercyclic. Since $\HC(w)$ is comeager whenever it is nonempty, if we let $M$ be a countable transitive model of a sufficient fragment of ZFC, with a Borel code for $b$ in $M$, and let $P$ be a perfect set of elements of $\ell^{2}$ which are mutually Cohen-generic over $M$, we get that $x \in \HC(b(y))$, for all distinct $x, y \in P$. The proof below carries out these ideas without using forcing.

\begin{lemma}\label{continuous function lemma}
There is a dense $G_\delta$ set $G \subseteq \ell^2$ and continuous function $f : G \rightarrow \{ 1 , 2 \}^\omega$ such that 
\begin{itemize}
\item $y \notin \HC (f(y))$ for each $y \in G$,
\item $\HC (f(y)) \neq \emptyset$ (and is therefore comeager) for each $y \in G$, and
\item for each open $U \subseteq \ell^2$, the image $f[U \cap G]$ contains at least two elements.  (As we will see in the proof, this actually follows from the first two properties.)
\end{itemize}
\end{lemma}

\begin{proof}
For each $y \in \ell^2$ and $n \in \omega$, let $i_{n , y} \in \omega$ be least such that
\begin{itemize}
\item $i_{n , y} > 2 i_{n - 1 , y}$ and
\item $\lVert y \upharpoonright [i_{n , y} , \infty) \rVert_2 < 2^{-1-n}$.
\end{itemize}
Let $f : \ell^2 \rightarrow \{ 1 , 2 \}^\omega$ be defined by 
\[
f(y) (i) = \begin{cases}
2 &\mbox{if } i = i_{n,y} \mbox{ for some } n \in \omega,\\
1 &\mbox{otherwise,}
\end{cases}\]
This function is Borel and hence (by Theorem 8.37 in Kechris \cite{kechris dst}), there is a dense $G_\delta$ set $G$ such that $f \upharpoonright G$ is continuous.

For each $y \in \ell^2$, the sequence $f(y) \in \{ 1 , 2 \}^\omega$ has infinitely many 2's and hence (by Lemma~\ref{infly many twos}) $\HC (f(y))$ is comeager.  In particular, $\HC (f(y))$ is comeager for all $y \in G$.

The next step is to see that $y \notin \HC (f(y))$.  Given $y$, fix any $k \geq i_{0,y}$.  Suppose that $i_{n,y} \leq k < i_{n+1 , y}$.  Since each $i_{p,y}$ is at least $2 i_{p-1 , y}$, it follows that no interval $I \subseteq [i_{n,y} , \infty)$ of length $k$ contains more that $n$ elements $i$ where $f(y) (i) = 2$.  Thus, to obtain $B^k_{f(y)} (y)$ from $y$, each bit of $y$ is multiplied by at most $n$ twos as it is shifted.  Therefore, 
\[
\lVert B^k_{f(y)} (y) \rVert_2 \leq 2^n \cdot \lVert y \upharpoonright [i_n , \infty) \rVert_2 \leq 1
\]
by the choice of $i_n$.  In particular, $B^k_{f(y)} (y) \in \{ z \in \ell^2 : \lVert z \rVert_2 > 1\}$ for only finitely many $k$ and hence $y \notin \HC (f(y))$.

Thus, the first two conditions in the statement of the Lemma have been satisfied.  To prove the third, fix an open set $U \subseteq \ell^2$ and pick any $y \in U \cap G$.  Since $\HC (f(y)) \neq \emptyset$ and is thus comeager, there exists $x \in \HC (f(y)) \cap U \cap G$.  Since $x \notin f(x)$, the weight sequences $f(x)$ and $f(y)$ must be distinct as otherwise
\[
x \in \HC (f(y)) = \HC (f(x))
\]
contrary to the choice of $f$.  It follows that $f[ U \cap G ]$ contains at least two elements.
\end{proof}


The key lemma in the proof of Theorem~\ref{T1} is Lemma~\ref{perfect set lemma} below.  In essence, the proof of Lemma~\ref{perfect set lemma} is a direct construction of a perfect set of mutually generic Cohen reals.  (Recall that a {\em perfect} set is a closed set with no isolated points.)  This will be essential to building a set $W \subseteq \{ 1 , 2 \}^\omega$ of weight sequences such that $\HC^* (W)$ is of arbitrary definable complexity.  For the present purposes, the salient property of perfect sets is that, for any pointclass $\Gamma$, a given perfect set has subsets which are not in $\Gamma$. 

\begin{lemma}\label{perfect set lemma}
Given $f : G \rightarrow \{ 1 , 2 \}^\omega$ as in Lemma~\ref{continuous function lemma}, there is a perfect set $P \subseteq \ell^2$ such that
\begin{itemize}
\item $x \notin \HC (f(x))$ for all $y \in P$ (this is from Lemma~\ref{continuous function lemma})
\item $x \in \HC (f(y))$ for all distinct $x,y \in P$.
\end{itemize}
\end{lemma}

\begin{proof}
First of all, let $f : \ell^2 \rightarrow \{ 1 , 2 \}^\omega$ be as in Lemma~\ref{continuous function lemma}, i.e., 
\begin{itemize}
\item $f$ is continuous on a comeager $G_\delta$ set $G \subseteq \ell^2$
\item $y \notin \HC (f(y))$ for each $y \in G$
\item $\HC (f(y))$ is comeager for each $y \in G$
\item $f$ is not constant on $U \cap G$ for each open $U \subseteq \ell^2$
\end{itemize}
Let $\HC = \{ (y,w) \in \ell^2 \times \{ 1 , 2 \}^\omega : y \in \HC (w)\}$.  Since $\HC$ and $G$ are both dense $G_\delta$ sets, there exist dense open sets $D_n \subseteq \ell^2 \times \{ 1 , 2 \}^\omega$ such that
\[
\HC \cap (G \times \{ 1 , 2 \}^\omega) \supseteq \bigcap_n D_n.
\]
and $D_0 \supseteq D_1 \supseteq D_2 \supseteq \ldots$.

The key is to build Cantor schemes $(P_s)_{s \in 2^{<\omega}}$ (in $\ell^2$) and $(Q_s)_{s \in 2^{<\omega}}$ (in $\{ 1 , 2 \}^\omega$) such that
\begin{itemize}
\item Each $P_s$ is open and $Q_s$ is a basic clopen set.
\item If $s$ is an initial segment of $t$, then $P_s \supseteq \overline P_t$ and $Q_s \supseteq Q_t$.
\item If $s$ and $t$ have no common extensions, then $\overline P_s \cap \overline P_t = \emptyset$ and $Q_s \cap Q_t = \emptyset$.
\item $f[\overline P_s \cap G] \subseteq Q_s$.
\item If $s,t \in 2^n$ are distinct, then $P_s \times Q_t \subseteq D_n$.
\end{itemize}
The construction is by induction on $|s|$.  Suppose that $P_s$ and $Q_s$ are given with the properties above for all $s \in 2^n$.  

{\em Step 1.}  By the properties of $f$ (Lemma~\ref{continuous function lemma}) each $f[P_s \cap G]$ has at least two elements.  Therefore, for all $s\in 2^n$, the continuity of $f \upharpoonright G$ implies that there are open sets $U_{s {}^\smallfrown 0}$ and $U_{s {}^\smallfrown 1}$ such that
\begin{itemize}
\item $\overline U_{s {}^\smallfrown 0} , \overline U_{s {}^\smallfrown 1} \subseteq P_s$
\item $f[U_{s {}^\smallfrown 0} \cap G] \cap f[U_{s {}^\smallfrown 1} \cap G] = \emptyset$
\end{itemize}
The second property above (together with the properties of the $Q_s$) implies that for all distinct $s,t \in 2^{n+1}$, the sets $f[U_s \cap G]$ and $f[U_t \cap G]$ are disjoint.

Now fix $s \in 2^{n+1}$ and pick $x \in U_s \cap G$.  By the properties of $f$, the set of hypercyclic vectors for $f(x)$ is comeager and hence has nonempty intersection with all $U_t \cap G$ (for $t \in 2^{n+1}$), i.e., $\left( (U_t \cap G) \times \{ f(x) \} \right) \cap \HC \neq \emptyset$.  In particular, shrinking the neighborhoods $U_t$ if necessary (for $t \neq s$), there must be a neighborhood $Q_s$ of $f(x)$ such that $U_t \times Q_s \subseteq D_{n+1}$ for each $t \in 2^{n+1}$ with $t \neq s$.  Finally, by the continuity of $f \upharpoonright G$, it is possible to shrink $U_s$ to ensure that $f[\overline U_s \cap G] \subseteq Q_s$.

Repeat the process above for each $s \in 2^{n+1}$.  Each time this process is repeated, a new $Q_s$ is produced and each $U_t$ (for $t \in 2^{n+1} \setminus \{ s \} $) shrinks finitely many times.  To complete the induction, let $P_s = U_s$ (after it has been shrunk as above).  The $P_s$ and $Q_s$ now satisfy the desired properties above.

\vspace{1em}

{\em Step 2.}  Let $P$ be the perfect set associated with the Cantor scheme $(P_s)_{s \in 2^{< \omega}}$, i.e., 
\[
P = \bigcup_{\alpha \in 2^\omega} \bigcap_n P_{\alpha \upharpoonright n}.
\]
It follows from the definition of $f$ that $y \notin \HC (f(x))$ for each $x \in P$.  On the other hand, suppose that $x , y \in P$ are distinct.  Let $\alpha , \beta \in 2^\omega$ be such that
\[
\{ x \} = \bigcap_n P_{\alpha \upharpoonright n} \qquad \mbox{and} \qquad \{ y \} = \bigcap_n P_{\beta \upharpoonright n}.
\]
If $n \in \omega$ is large enough that $\alpha \upharpoonright n \neq \beta \upharpoonright n$, the properties of the $P_s$ and $Q_s$ guarantee that
\[
(x , f(y)) \in P_{\alpha \upharpoonright n} \times Q_{\beta \upharpoonright n} \subseteq D_n.
\] 
Since this holds for all but finitely many $n$, it follows that $(x , f(y)) \in \HC$, i.e., $x \in \HC (f(y))$ as desired.  This completes the proof.
\end{proof}

\begin{proof}[Proof of Theorem~\ref{T1}]
Let $f : P \rightarrow \{ 1 , 2 \}^\omega$ be as in Lemma~\ref{perfect set lemma}, i.e., for all $x,y \in P$,
\begin{itemize}
\item $y \notin \HC (f(y))$ and
\item $x \in \HC (f(y))$ if $x \neq y$.
\end{itemize}
Let $\Gamma$ be any pointclass as in the statement of Theorem~\ref{T1}.  Choose any $A \subseteq P$ with $A \in \Gamma \setminus \bar \Gamma$.  Consider the common hypercyclic vectors of $f[A]$, i.e., the set $\HC^* (f[A])$.  For $x,y \in P$,
\[
x \in \HC (f(y)) \iff x \neq y
\]
and so $\HC^* (f[A]) \cap P = P \setminus A$.  In particular, $\HC^* (f[A]) \cap P \notin \Gamma$ and hence $\HC^* (f[A]) \notin \Gamma$ either since $\Gamma$ contains the closed subsets of $\ell^2$.

This completes the proof.
\end{proof}

\section{Proof of Theorem~\ref{T2}}

Given a countable set $A$, a subset $a$ of $A$ may be identified with with its characteristic function in $2^A$.  In what follows, we will freely make use of this identification and regard $\mathcal P (A)$ (the power set of $A$) as being equipped with the usual product-of-discrete topology of $2^A$.

The following technical lemma is the key to the proof of Theorem~\ref{T2}.  In fact, it can be used to prove Theorem~\ref{T1} as well.

\begin{lemma}\label{L1}
Given a countable set $A$, it is possible to assign to each $a \subseteq A$, sequences $y_a \in \ell^2$ and $w_a \in \{ 1,2 \}^\omega$ such that
\begin{enumerate}
\item for all $a, b \subseteq A$, we have $y_a \in \HC (w_b) \iff b \nsubseteq a$, and
\item the maps $a \mapsto y_a$ and $a \mapsto w_a$ are homeomorphisms between $2^A$ and their ranges.
\end{enumerate}
\end{lemma}

\begin{proof}
Let $\pi \colon \omega \to \bbQ^{\less\omega}$ be a surjection.  Let $A$ be the fixed countable set from the statement of the lemma.  For coding purposes, fix a bijection
\[
\langle \cdot , \cdot , \cdot\rangle :  \omega \times (\mathbb Q \cap (0,1)) \times A \rightarrow \omega.
\]
Given $n \in \omega$, let $p_n \in \omega$, $\varepsilon_n > 0$ and $i_n \in A$ be such that
\[
n = \langle p_n , \varepsilon_n , i_n \rangle.
\]
Finally, let
\[
\rho_n = \min \{ \varepsilon_r : r \leq n\}.
\]

The first step of the proof is to choose a suitable partition
\[
I_0 , J_0 , I_1 , J_1 , \ldots
\]
of $\omega$ into consecutive intervals, i.e., such that $\min (J_n) = \max (I_n) + 1$ and $\min (I_{n+1}) = \max (J_n) + 1$.  For convenience, we let $I_0 = \{ 0 \}$.  Each $J_n$ will be chosen with $|J_n| = | \pi (p_n)|$.  The lengths of the $I_n$ (for $n> 0$) will be chosen recursively and, for concreteness, of minimal length satisfying
\begin{enumerate}
	\item $|I_n| \geq |I_{n-1}|$,
	\item $| I_n | > \max (J_{n-1})$ and
	\item $2^{-| I_n|} \cdot \lVert \pi (p_n) \rVert_2 \leq 2^{-n-1} \cdot \rho_n \cdot 2^{- \max(J_{n-1})} \cdot 2^{-|I_{n-1}|}$.
\end{enumerate}
for $n > 1$.  The length of $I_0$ is arbitrary -- $I_0$ can even be the empty interval.

The next step is to define the desired $y_a$ and $w_a$ for each $a \subseteq A$.  For $n = \langle p , \varepsilon , i\rangle$, define $y_a$ on $I_n$ and $J_n$ by
\begin{enumerate}
	\item $y_a \upharpoonright I_n = \bar 0$,
	\item $y_a \upharpoonright J_n = \bar 0$ if $i \in a$, and
	\item $y_a \upharpoonright J_n = 2^{-| I_n |} \cdot \pi (p)$ if $i \notin a$.
\end{enumerate}
The first important observation about the map $a \mapsto y_a$ is that it is continuous.  To see this, first observe that every initial segment of $y_a$ is determined by an initial segment of $a$.  This implies that $a \mapsto y_a$ is continuous into the product topology on $\ell^2$ (which it inherits from $\mathbb R^\omega$).  Now invoke Lemma \ref{L0} and use the fact that $y_a$ is always termwise bounded by $y_\emptyset \in \ell^2$.  It now follows that $a \mapsto y_a$ is in fact continuous with respect to the norm-topology on $\ell^2$.

It also follows from the definition of $y_a$ that the function $a \mapsto y_a$ is injective.  As the domain of this map ($2^A$) is compact, $a \mapsto y_a$ must therefore be a homeomorphism with its range.

%

Now define $w_a \in \{ 1,2 \}^\omega$ (for $a \subseteq A$) by making sure that the restrictions $w_a \upharpoonright I_n \cup J_n$ satisfy
\begin{enumerate}
	\item $(\forall n) ( i_n \notin a \implies w_a \upharpoonright I_n \cup J_n = \bar 1)$,
	\item $(\forall n) ( i_n \in a \implies (\forall j \in J_n)( | \{ t \in [ j , \min (J_n) + j) : w_a (t) = 2\} | = | I_n |)$ and
	\item if $i , j \in I_n$ with $i < j$ and $w_a (j) = 2$, then $w_a (i) = 2$.
\end{enumerate}
The continuity of $a \mapsto w_a$ follows from the fact that initial segments of $w_a$ are completely determined by initial segments of $a$.

The next three claims will complete the proof.  The proofs of these three claims all follow similar arguments using the definitions of the $y_a$ and $w_a$.

\vspace{1em}

\begin{claim*}
Each $y_a$ is in $\ell^2$.
\end{claim*}

It suffices to show that the $\ell^2$ norm of $y_a$ is finite.  Indeed, by the triangle inequality and the third part of the definition of $y_a$,
\begin{align*}
\lVert y_a \rVert_2
&\leq \sum_{n \in \omega} \lVert y_a \upharpoonright J_n \rVert_2 \\
&\leq \sum_{n \in \omega} 2^{-| I_n |} \cdot \lVert  \pi ( p_n ) \rVert_2 \\
&\leq \sum_{n \in \omega} 2^{-n-1} \cdot \rho_n \cdot 2^{- \max(J_{n-1})} \cdot 2^{-|I_{n-1}|} \\
&\leq \sum_{n \in \omega} 2^{-n-1}\\
&\leq 1
\end{align*}
This proves the claim.

\vspace{1em}

\begin{claim*}
If $a,b \subseteq A$ with $b \subseteq a$, then $y_a \notin \HC (w_b)$.
\end{claim*}

For this claim, it suffices to show that $\lVert B_{w_b}^k (y_a) \rVert_2 \leq 1$ or $B_{w_b}^k (y_a) (0) = 0$ for each $k \in \omega$.  This will establish that there is no $k \in \omega$ such that $B_{w_b}^k (y_a)$ is in the open set
\[
U = \{ y \in \ell^2 : \lVert y \rVert_2 > 1 \mbox{ and } y(0) \neq 0\}.
\]
To this end, fix $k \in \omega$ and let $n \in \omega$ be such that $k \in I_n \cup J_n$.  First of all, if $i_n \in a$, then $y_a \upharpoonright I_n \cup J_n = \bar 0$ and hence
\[
B_{w_b}^k (y_a) (0) = w_b (0) \cdot \ldots \cdot w_b (k-1) \cdot y_a (k) = 0.
\]
On the other hand, if $i_n \notin a \supseteq b$, then $w_b \upharpoonright I_n \cup J_n = \bar 1$ and hence
\[
| \{ j < k : w_b (j) = 2 \}| \leq \max (J_{n-1}).
\]
To obtain an estimate of $\lVert B^k_{w_b} (y_a) \rVert_2$, a couple preliminary observations will be useful.  Suppose $t \in \omega$ is such that $k+t \in I_r$ for some $r \in \omega$.  In this case,
\[
B^k_{w_b} (y_a) (t) = 0
\]
since $y_a (k+t) = 0$.  If $k+t \in J_n$ (where $k \in I_n \cup J_n$), then
\[
|B^k_{w_b} (y_a) (t)| \leq 2^{\max (J_{n-1})} \cdot | y_a (k+t)|
\]
since $w_b \upharpoonright I_n \cup J_n = \bar 1$.  Finally, if $k+t \in J_r$ for some $r > n$, then
\begin{align*}
|B^k_{w_b} (y_a) (t)|
&\leq 2^k \cdot | y_a (k+t)|\\
& \leq 2^{\max (J_{r-1})} \cdot | y_a (k+t)|
\end{align*}
since $k \leq \max (J_n) \leq \max (J_{r-1})$.  It now follows by the triangle inequality that
\begin{align*}
\lVert B_{w_b}^k (y_a) \rVert_2
&\leq \sum_{r \geq n} 2^{\max (J_{r-1})} \cdot \lVert y_a \upharpoonright J_{r}\rVert_2\\
&\leq \sum_{r \geq n} 2^{\max (J_{r-1})} \cdot 2^{-r-1} \cdot \rho_r \cdot 2^{-\max(J_{r-1})} \cdot 2^{-|I_{r-1}|}\\
&\leq \sum_{r \geq n}  2^{-r-1}\\
&\leq 1
\end{align*}
This completes the proof of the claim.

\vspace{1em}

\begin{claim*}
If $a,b \subseteq A$ with $b \nsubseteq a$, then $y_a \in \HC (w_b)$.
\end{claim*}

For this final claim, it suffices to show that, for each $q \in \mathbb Q^{< \omega}$ and $\varepsilon > 0$, there is a $k \in \omega$ such that $B_{w_b}^k (y_a)$ is in the open set
\[
U_{q , \varepsilon} = \{ x \in \ell^2 : \lVert (x \upharpoonright |q|) - q \rVert_\infty < \varepsilon |q|^{-1/2}  \mbox{ and } \lVert x \upharpoonright [ |q| , \infty ) \rVert_2 < \varepsilon\}
\]
as these open sets form a topological basis for $\ell^2$.  Indeed, fix $q \in \mathbb Q^{<\omega}$ and let $p \in \omega$ be such that $\pi (p) = q$.  Fix $i \in b \setminus a$ and let $n = \langle p , \varepsilon, i \rangle$.  Since $i \in b$ and $i \notin a$, the second case in the definition of $w_b \upharpoonright I_n \cup J_n$ and the second case in the definition of $y_a \upharpoonright J_n$ are active.  In particular, for each $j \in J_n$,
\[
| \{ t \in [j , \min (J_n) + j) : w_b (t) = 2 \} | = | I_n |.
\]
It follows that
\[
B_{w_b}^{ \min( J_n ) } (y_a) = \pi (p) {}^\smallfrown y
\]
for some $y \in \ell^2$.  To show that $B_{w_b}^{ \min (J_n) } (y_a) \in U_{q , \varepsilon}$, it now suffices to show that $\lVert y \rVert_2 < \varepsilon$, since $q \prec B_{w_b}^{ \min (J_n) } (y_a)$ by the choice of $n$.  Indeed, observe that, again by the triangle inequality,
\begin{align*}
\lVert y \rVert_2
&\leq 2^{\min (J_n)} \cdot \sum_{r > n} \lVert y_a \upharpoonright J_r\rVert_2\\
&\leq 2^{\min (J_n)} \cdot \sum_{r > n} 2^{-|I_r|} \cdot \lVert \pi (p_r) \rVert_2 \\
&\leq 2^{\min (J_n)} \cdot \sum_{r > n} 2^{-r-1} \cdot \rho_r \cdot 2^{- \max(J_{r-1})} \cdot 2^{-|I_{r-1}|} \\
&\leq 2^{\max (J_n)} \cdot \sum_{r > n} 2^{-r-1} \cdot \rho_n \cdot  2^{-\max (J_n)} \\
&\leq \varepsilon \cdot \sum_{r > n} 2^{-r-1}\\
&< \varepsilon
\end{align*}
since $\rho_n \leq \varepsilon = \varepsilon_n$.  This complete the proof of the claim and proves Lemma~\ref{L1}.
\end{proof}

\begin{proof}[Proof of Theorem~\ref{T2}]
The key to this proof is an application of Lemma~\ref{L1} with the countable set $A$ taken to be $\omega^{<\omega}$.  With this in mind, let
\[
{\sf Wf} = \{ T \subseteq \omega^{<\omega} : T \mbox{ is a well-founded subtree}\}
\]
and
\[
{ \sf C} = \{ p \subseteq \omega^{<\omega} : p \mbox{ is a maximal $\prec$-chain}\}.
\]
In other words, ${\sf C}$ may be identified with the set of infinite branches through $\omega^{<\omega}$.  The set ${\sf Wf}$ properly co-analytic while ${\sf C}$ is $G_\delta$.  Let $W = \{ w_p : p \in {\sf C}\}$ and notice that $W$ is also $G_\delta$ since $p \mapsto w_p$ is a homeomorphism by Lemma~\ref{L1}.  By Proposition~\ref{hc* is co-analytic},
\[
\HC^* (W) = \bigcap_{w \in W} \HC (w)
\]
is co-analytic since $W$ is Borel.  To see that it is not analytic, recall that any subtree $T \subseteq \omega^{<\omega}$ is well-founded iff $T$ has no infinite branches.  In turn, this is equivalent to
\begin{align*}
(\forall p \in {\sf C}) (p \nsubseteq T)
&\iff (\forall p \in {\sf C}) (y_T \in \HC (w_p)) && \mbox{(by Lemma~\ref{L1})}\\
&\iff (y_T \in \HC^* (W).
\end{align*}
It follows that ${\sf Wf}$ is a continuous preimage of $\HC^* (W)$ under the map $T \mapsto y_T$.  In turn, this implies that $\HC^* (W)$ cannot be analytic.
\end{proof}

As was mentioned above, Lemma~\ref{L1} may be used to give another proof of Theorem~\ref{T1}.

\begin{proof}[Alternative proof of Theorem~\ref{T1}]
Let $P \subseteq 2^\omega$ be a perfect set such that $b \nsubseteq a$ for any two distinct $a , b \in P$.  The construction of such a set is a standard inductive argument (similar to the construction of a perfect independent set).  Let $y_a$ and $w_a$ be as in Lemma~\ref{L1} for $a \subseteq \omega$.  It follows that $y_a \in \HC (w_b)$ iff $a \neq b$ for all $a,b \in P$.

Given a non-self-dual pointclass $\Gamma$ which contains both the open and closed sets, fix $Y \subseteq P$ with $Y \in \Gamma \setminus \bar \Gamma$.  Since $P$ is closed, it follows that $P \setminus Y \in \bar \Gamma \setminus \Gamma$.  Let
\[
W = \{ w_a : a \in Y \}.
\]
Now consider the set
\[
\HC^* (W) = \bigcap_{w \in W} \HC (w).
\]
For $a \in P$, notice that $y_a \in \HC^* (W)$ iff $a \notin Y$.  Hence,
\[
\HC^* (W) \cap \{ y_a : a \in P\} = \{ y_a : a \in P \mbox{ and } a \notin Y\} = \{ y_a : a \in P \setminus Y \}
\]
It follows that $\HC^* (W) \notin \Gamma$ since $\{ y_a : a \in P \}$ is closed and $\{ y_a : a \in P \setminus Y \} \in \bar \Gamma \setminus \Gamma$ (because $a \mapsto y_a$ is a homeomorphism).  This completes the proof of the theorem.
\end{proof}

\section{Proof of Theorem~\ref{non baire under ch}}

The goal of this section is to show that, assuming Martin's Axiom (MA), there is a set $W \subseteq \{ 1 , 2\}$ whose set of common hypercyclic vectors does not have the property of Baire (Theorem~\ref{non baire under ch}). We expect that the assumption of additional axioms is unnecessary.

\subsection{Nicely hypercyclic vectors}

\begin{definition} Given $n,k \in \omega$, we say that function $w$ with domain $\omega$ is $n$-\emph{nice} at $k$ if $w(i) = w(k+i)$ for all $i < n$.
\end{definition}

Recall that for $q \in \bbQ^{<\omega}$ and $\epsilon \in \bbQ^{+}$ we have defined the set \[
U_{q , \varepsilon} = \{ x \in \ell^2 : \lVert (x \upharpoonright |q|) - q \rVert_\infty < \varepsilon |q|^{-1/2}  \mbox{ and } \lVert x \upharpoonright [ |q| , \infty ) \rVert_2 < \varepsilon\}
\]
and that the collection of sets of the form $U_{q, \epsilon}$ forms a basis for $\ell^{2}$.

\begin{definition}\label{mapsnicelydef} Given $k \in \omega$, $w \in \bbR^{\omega}$, $y \in \ell^{2}$, $q \in \bbQ^{<\omega}$ and $\epsilon \in \bbQ^{+}$, we say that $B^{k}_{w}$ \emph{maps} $y$ \emph{nicely} into $U_{q, \epsilon}$ if
\begin{enumerate}
\item $B^{k}_{w}(y) \in U_{q,\epsilon}$;
\item $w$ is $|q|$-nice at $k$;
\item\label{cthree} $\norm{y \restrict [k + |q|, \infty)}_{2} < \epsilon 2^{-k}$.
\end{enumerate}
We say that $y$ is \emph{nicely hypercyclic} for $w$ ($y \in \NHC(W)$) if
for each $U_{q,\epsilon}$ there is a $k$ such that $B^{k}_{w}$ maps $y$ nicely into $U_{q, \epsilon}$, and
\emph{nicely hypercyclic} if there is such a $w$. We write $\sf NH$ for the set of nicely hypercyclic elements of $\ell^{2}$. 
\end{definition}

\begin{remark}\label{cthreerem}
  Condition (\ref{cthree}) of Definition \ref{mapsnicelydef} implies
  that $\norm{B^{k}_{v}(y) \restrict [|q|, \infty)}_{2} < \epsilon$ for any $v \in \{1,2\}^{\omega}$, which is the second condition in the statement $B^{k}_{v}(y) \in U_{q,\epsilon}$. Given that $y$ satisfies Condition (\ref{cthree}), then, satisfaction of the other two conditions depends only finite initial segments of $w$ and $y$. In particular, if $w \in \{1,2\}^{\omega}$,  $B^{k}_{w}$ maps $y$ nicely into $U_{q, \epsilon}$ and $v \in \{1,2\}^{\omega}$ is such that
$|w^{-1}[\{2\}] \cap k| = |v^{-1}[\{2\}] \cap k|$ and $v$ is $|q|$-nice at $k$, then $B^{k}_{v}$ maps $y$ nicely into $U_{q, \epsilon}$.
\end{remark}

We present a convenient sufficient condition for being nicely hypercyclic and use it to show that $\sf NH$ is comeager in $\ell^{2}$.

\begin{lemma}
  An element $y$ of $\ell^{2}$ is nicely hypercyclic if for each $U_{q,\epsilon}$ and each $m\in \omega$ there exist $n \leq k$ in $\omega$ with $n \geq m$ and $k-n \geq m$ such that
  \[\norm{ (2^{n}y \restrict [k, k + |q|)) - q}_{\infty} < \epsilon |q|^{-1/2}\] and
  \[\norm{y \restrict [k + |q|, \infty)}_{2} < \epsilon 2^{-k}.\]
\end{lemma}

\begin{proof}
  Assuming that $y$ satisfies the given condition, recursively build a $w \in \{1,2\}^{\omega}$ for which $y$ is nicely hypercyclic. Given a finite initial segment $w \restrict m$ of the desired $w$, and a basic open set $U_{q, \epsilon}$, let $n \leq k$ be as given by the hypothesis. Extend $w \restrict m$ to $w \restrict (k + |q|)$ so that $|\{ i < k : w(i) = 2\}| = n$ and $w(i) = w(k + i)$ for all $i < |q|$.
  These choices suffice to meet the challenge given by $U_{q,\epsilon}$.
\end{proof}

\begin{lemma}
The set of nicely hypercyclic vectors is comeager in $\ell^{2}$.
\end{lemma}

\begin{proof}
By the previous lemma, it suffices to show that for each $U_{q,\epsilon}$ and $m \in \omega$, for comeagerly many $y$ there exist $n \leq k$ with $n \geq m$ and $k-n \geq m$ such that \[\norm{ 2^{n}y \restrict [k, k + |q|) - q}_{\infty} < \epsilon |q|^{-1/2}\] and
\[\norm{y \restrict [k + |q|, \infty)}_{2} < \epsilon 2^{-k}.\] In fact we will show that the set of such $y$ is dense open. Fix $q, \epsilon$ and $m$, and a basic open set
$U_{r, \delta}$.

Let
\begin{itemize}
\item $n$ and $k$ be such that $n \geq m$, $2^{n} > 3\norm{q}_{2}/\delta$, $k \geq n + m$ and $k > |r|$;
\item $s$ be $r^{\frown}t^{\frown}(2^{-n}q)$, where $t$ is the all-$0$ sequence of length $k - |r|$;
\item $\rho> 0$ be less than $\epsilon 2^{-k}$ and $\delta/3$.
\end{itemize}

Now let $z$ be an element of $U_{s, \rho}$. Then
\[\norm{ (z \restrict |r|) - r}_{\infty} \leq \norm{ (z \restrict |s|) - s}_{\infty} < \rho |s|^{-1/2} < \delta |r|^{-1/2}\]
 and
\begin{align*}
\norm{z \restrict [|r|, \infty)}_{2} &\leq \norm{z \restrict [|r|, |s|)}_{2} + \norm{z \restrict [|s|, \infty)}_{2}\\
 &< \norm{s \restrict [|r|, |s|)}_{2} + \norm{(z \restrict [|r|, |s|)) - (s \restrict [|r|, |s|))}_{2} + \rho\\
 &\leq 2^{-n}\norm{q}_{2} + (|s|-|r|)^{1/2}\norm{(z \restrict [|r|, |s|)) - (s \restrict [|r|, |s|))}_{\infty} + \rho\\
 &< 2^{-n}\norm{q}_{2} + (|s|-|r|)^{1/2}\rho|s|^{-1/2} + \rho\\
 &< \delta/3 + 2\rho < \delta,
\end{align*}
 so $z \in U_{r, \delta}$. Furthermore,
\begin{align*}
\norm{z \restrict [k, k+|q|) - 2^{-n}q}_{\infty} &= \norm{z \restrict [k, k+|q|) - s \restrict[k, k + |q|)}_{\infty}\\
&\leq \norm{(z \restrict (k + |q|) - s}_{\infty}\\
& < \rho (k + |q|)^{-1/2},
\end{align*}
so
\[\norm{ 2^{n}z \restrict [k, k + |q|) - q}_{\infty} < 2^{n}\rho(k + |q|)^{-1/2} < \epsilon |q|^{-1/2}.\]
Since \[\norm{z \restrict [k + |q|, \infty)}_{2} < \rho <  \epsilon 2^{-k},\]
$z$, and thus every member of $U_{s, \rho}$, satisfies the given conditions on $q, \epsilon$ and $m$.

\end{proof}

We define the set of nice weight sequences by adding a nontriviality condition.

\begin{definition} We say that a $w \in \{1,2\}^{\omega}$ is \emph{nice} if $w^{-1}[\{1\}]$ and $w^{-1}[\{2\}]$ are both infinite and, for each $n \in \omega$ there
are infinitely many $k \in \omega$ such that $w$ is $n$-nice at $k$.
\end{definition}

The following remark is not used in the rest of the section, and is stated for its independent interest. 

\begin{remark} If $W$ is a countable set of
nice weight sequences then $\NHC^{*}(W)$ is nonempty. To see this, recursively build a $y \in \NHC^{*}(W)$ by choosing initial segments $y \restrict m$ while making promises to keep $\norm{y \restrict [m, \infty)}_{2}$ smaller than some sufficiently small $\rho_{m}$, letting $\rho_{0} =1$. Given $y \restrict m$, $\rho_{m}$, $w \in W$ and a basic open set $U_{q, \epsilon}$, let $n$ and $k$ be such that \begin{itemize}
\item $2^{-n}\norm{q}_{2} < \rho_{m}$,
\item $k \geq m$,
\item $|w^{-1}[\{2\}]| = n$ and
\item $w$ is $|q|$-nice at $k$.
\end{itemize}
Then we can extend $y$ with all zeros until position $k$, and let $y \restrict [k, k + |q|) = 2^{-n}q$. We then set $\rho_{k + |q|}$ to be smaller than both $\epsilon 2^{-k}$ and $\rho_{m} - 2^{-n}\norm{q}_{2}$ and continue the construction.
\end{remark}

\subsection{The poset of partial witnesses}

As a notational convenience, for each finite sequence $\alpha$ in $\{1,2\}^{\less\omega}$, let $\alpha^{+}$ be the infinite extension of $\alpha$ with all $2$'s.

Consider the collection of all triples $( \alpha , r , \delta ) \in \{ 1 , 2 \}^{<\omega} \times \mathbb Q^{<\omega} \times (\mathbb Q \cap (0,1) )$ (henceforth referred to as {\em conditions}) satisfying the following:
\begin{itemize}
\item $| \alpha | \geq | r |$ and
\item $\delta < 2^{-| \alpha |}$.
\end{itemize}
Note that these two properties together imply that
\[
B^{|\alpha |}_{\alpha^+} [ U_{r , \delta} ] \subseteq \{ y \in \ell^2 : \lVert y \rVert_2 < 1\}.
\]
Given conditions $p_1 = ( \alpha_1 , r_1 , \delta_1)$ and $p_2 = (\alpha_2 , r_2 , \delta_2)$, say that $p_2$ {\em extends} $p_1$ (written $p_2 < p_1$) iff
\begin{itemize}
\item $\alpha_1$ is an initial segment of $\alpha_2$,
\item $\overline U_{r_2 , \delta_2} \subseteq U_{r_1 , \delta_1}$ and,
\item for all $k \in [ | \alpha_1 | , | \alpha_2 | )$,
\[
B^k_{\alpha_2^+} [U_{r_2 , \delta_2} ] \subseteq \{ y \in \ell^2 : \lVert y \rVert_2 < 1\}.
\]
\end{itemize}

The next two lemmas are the key to using either CH or Martin's Axiom to construct a set $W \subseteq \{ 1  , 2 \}^\omega$ of weight sequences such that $\HC^* (W)$ does not have the property of Baire.

\begin{lemma}\label{avoid a comeager}
Given a condition $p_1 = ( \alpha , r , \delta )$, a dense open set $D \subseteq \ell^2$ and a natural number $N$, there exists a condition $p_2 = (\beta , s , \eta)$ such that
\begin{enumerate}
\item $p_2 < p_1$
\item $U_{s , \eta } \subseteq D$
\item $| \beta | > N$
\end{enumerate}
\end{lemma}

\begin{proof}
Fix $p_1 = (\alpha , r , \delta)$, a dense open set $D \subseteq \ell^2$ and $N \in \omega$.  If necessary, extend $p_1$ by shrinking the neighborhood $U_{r , \delta}$ and assume that $\delta < 2^{-| \alpha | - 1}$.

Since $D$ is dense, choose a basic open set $U_{s , \eta}$ with
\[
\overline U_{s , \eta} \subseteq U_{r , \delta} \cap D.
\]
Note that
\[
B^{|\alpha |}_{\alpha^+} [ U_{r , \delta} ] \subseteq \{ z \in \ell^2 : \lVert z \rVert_2 < 1\}
\]
(by the definition of ``condition'') and hence
\[
B^{|\alpha |}_{\alpha^+} [ U_{s , \eta}] \subseteq \{ z \in \ell^2 : \lVert z \rVert_2 < 1\}
\]
as well.  Without loss of generality, assume that $|s| \geq |\alpha |$ and $\eta < 2^{-| \beta | - 1}$.  Let $\beta \in \{ 1 , 2 \}^{< \omega}$ have length greater than $\max \{ N , |s| \}$ and be of the form
\[
\alpha {}^\smallfrown 1 \, 1 \ldots 1.
\]
Now set $p_2 = (\beta , s , \eta)$.  To check that $p_2 < p_1$ it remains only to show that 
\[
B^k_{\beta^+} \left[ U_{s , \eta } \right] \subseteq \{ y \in \ell^2 : \lVert y \rVert_2 < 1\}
\]
for all $k \in [ | \alpha | , | \beta | )$.  Towards this end, fix $x \in U_{s , \eta}$.  First note that
\[
\lVert x \upharpoonright [ | \alpha | , \infty ) \rVert_2 < \delta < 2^{- | \alpha | - 1}
\]
since $x \in U_{s , \eta} \subseteq U_{r , \delta}$.  Second, observe that 
\[
\lVert x \upharpoonright [ | \beta |  , \infty ) \rVert_2 < \eta < 2^{- | \beta | - 1}
\]
by the definition of $U_{s , \eta}$, the choice of $\eta$ and the fact that $| \beta | > |s|$.  By the triangle inequality and the fact that $\beta^{-1} [\{ 2\}] = \alpha^{-1} [\{ 2 \}]$, it now follows that
\begin{align*}
\lVert B^k_{\beta^+} (x) \rVert_2
&\leq 2^{| \alpha |} \cdot \lVert x \upharpoonright [ | \alpha | , | \beta |) \rVert_2 
\  + \  2^{| \beta |} \cdot  \lVert x \upharpoonright [ | \beta | , \infty) \rVert_2\\
&< 2^{| \alpha |} \cdot 2^{- | \alpha | - 1} + 2^{| \beta |} \cdot 2^{- | \beta | - 1}\\
&= 1.
\end{align*}
This shows that $B^k_{\beta^+} \left[ U_{s , \eta} \right] \subseteq \{ y \in \ell^2 : \lVert y \rVert_2 < 1\}$ and completes the proof of the lemma.
\end{proof}

\begin{lemma}\label{nice conditions are dense}
Given a condition $p_1$, a nicely hypercyclic $y \in \ell^2$ and a basic neighborhood $U_{q , \varepsilon}$, there is a condition $p_2 = (\beta , s , \eta)$ such that $p_2 < p_1$ and $B^k_{\beta^+}$ maps $y$ nicely into $U_{q , \varepsilon}$ for some $k \leq | \alpha |$.
\end{lemma}

\begin{proof}
Write $p_1 = (\alpha , r , \delta)$ and fix $v \in \{ 1 , 2 \}^\omega$ such that $y \in \NHC (v)$.  Using part 3 of Lemma~\ref{avoid a comeager}, we may extend $\alpha$ if necessary and assume that $| \alpha | \geq | q |$.  Now let $k > | \alpha |$ be such that $B^k_v$ maps $y$ nicely into $U_{q , \varepsilon}$, i.e., 
\begin{itemize}
\item $B^k_v (y) \in U_{q , \varepsilon}$,
\item $v \upharpoonright | q  |$ is an initial segment of $v \upharpoonright [k , \infty)$, and
\item $\lVert y \upharpoonright [ k + |q| , \infty ) \rVert_2 < \varepsilon 2^{-k}$.
\end{itemize}
Since there are arbitrarily large such $k$, it is safe to assume that $k$ is large enough to guarantee
\[
\left| v^{-1} [ \{ 1 \} ] \cap [ | \alpha | , k ) \right| \geq \left| v^{-1} [ \{ 2 \} ] \cap | \alpha | \right|.
\]
In turn, this means that $\alpha$ can be extended to $\alpha_1$ with the property that
\[
\left| \alpha_1^{-1} [ \{ 2 \} ] \cap k \right|  = \left| v^{-1} [ \{ 2 \} ] \right|.
\]
Let $\beta = \alpha_1 {}^\smallfrown \left( \alpha \upharpoonright | q | \right)$.  Since $B^k_v$ maps $y$ nicely into $U_{q , \varepsilon}$, it follows that $B^k_{\beta^+}$ does as well.

Now let $\eta = \min\left\{ \frac{\delta}{2} , 2^{-| \beta | - 1}\right\}$ and take
\[
p_2 = ( \beta , r , \eta).
\]
It follows that $p_2$ is a condition and $p_2 < p_1$.
\end{proof}

\vspace{5em}

Equipped with the lemmas above, the next step is to complete the proof of Theorem~\ref{non baire under ch}.

\begin{proof}[Theorem~\ref{non baire under ch}]
Assume that Martin's Axiom holds.  Let $\langle C_a : a < \mathfrak c \rangle$ enumerate all dense $G_\delta$ subsets of $\ell^2$.  The goal is to choose $y_a \in \ell^2$ and $w_a \in \{ 1 , 2 \}^\omega$ (for $a < \mathfrak c$) such that 
\begin{itemize}
\item $\HC (w_a) \supseteq \{ y _b : b < a \}$
\item $y_a \in \left( \bigcap_{b \leq a} \HC (w_b) \right) \cap C_a \cap {\sf NH}$
\item $C_a \setminus \HC (w_a) \neq \emptyset$
\end{itemize}
Before showing how to find such $w_a$ and $y_a$, the first step is to verify that the conditions above are sufficient to guarantee the existence of a set of common hypercyclic vectors without the property of Baire.  Indeed, let
\[
W = \{ w_a : a < \mathfrak c\}.
\]
Notice that the first and second conditions above show that $y_a \in \HC (w_b)$ for all $a,b < \mathfrak c$ and so
\[
\HC^* (W) \supseteq \{ y_a : a < \mathfrak c\}.
\]
Since $\HC^* (W)$ is closed under changes in finitely many coordinates, if $\HC^* (W)$ has the Baire property then it is either meager or comeager.  Therefore, it is sufficient to show that 
\begin{itemize}
\item $\HC^* (W)$ intersects every dense $G_\delta$ subset of $\ell^2$ (i.e., it is non-meager) and
\item $\HC^* (W)$ contains no dense $G_\delta$ subset of $\ell^2$ (i.e., it is not comeager).
\end{itemize}
The first of these statements is witnessed by the fact that $y_a \in C_a \cap \HC^* (W)$ for every $a < \mathfrak c$ (see the second condition above).  The second is a consequence of the third condition above, i.e., $C_a \setminus \HC (w_a) \neq \emptyset$ for all $a < \mathfrak c$.

The construction of the $y_a$ and $w_a$ proceeds in stages $a < \mathfrak c$.  Suppose that $y_b \in \ell^2$ and $w_b \in \{ 1 , 2 \}^\omega$ are given for all $b < a$ and satisfy the properties above.  Let $D_0 \supseteq D_1 \supseteq \ldots$ be dense open sets such that $C_a = \bigcap_n D_n$.

Let $(\mathbb P , <)$ be the set of conditions $(\alpha , r , \delta)$ ordered by the extension relation $<$.  First of all, $\mathbb P$ is ccc since it is countable, so Martin's Axiom applies.

For each $b < a$ and basic neighborhood $U_{q , \varepsilon}$, consider the set of conditions
\[
E_{b , q , \varepsilon} = \left\{ (\alpha , r , \delta)  \in \mathbb P : \left(\exists k \leq | \alpha |\right) \left(B^k_{\alpha^+} \mbox{ maps $y_b$ nicely into } U_{q , \varepsilon} \right) \right\}.
\]
It follows from Lemma~\ref{nice conditions are dense} that each $E_{b , q , \varepsilon}$ is dense.  For each $n \in \omega$ and $p \in \mathbb P$, define
\begin{align*}
F_{n , p} = \big\{ (\alpha , r , \delta) \in \mathbb P &: \left(  (| \alpha | \geq n) \wedge ( (\alpha , r , \delta) < p) \wedge U_{r , \delta} \subseteq D_n \right) \\
&\mbox{ or } (\alpha , r , \delta ) \mbox{ and } p \mbox{ have no common extensions}\big\}.
\end{align*}
Lemma~\ref{avoid a comeager} implies that each $F_{n , p}$ is dense in $\mathbb P$.

It is now possible to apply ${\rm MA} (|a|)$ to obtain a filter $G \subseteq \mathbb P$ such that $G$ has nonempty intersection with all of the dense sets above.  Since $G$ is a filter, if $(\alpha , r , \delta) , (\beta , s , \eta) \in G$, then either $\alpha \subseteq \beta$ or $\beta \subseteq \alpha$ as otherwise $G$ contains incompatible elements.  It thus follows that the weight sequence
\[
w_a = \bigcup \left\{ \alpha : (\exists r , \delta) ( (\alpha , r , \delta) \in G) \right\}
\]
is well-defined.

\noindent {\em Claim.}  $C_a \setminus \HC (w_a) \neq \emptyset$.

To prove this claim, choose a descending sequence $p_0 > p_1 > \ldots$ of conditions in $G$ by induction as follows:
\begin{itemize}
\item $p_0 \in G$ is arbitrary
\item Given $p_0 > p_1 > \ldots > p_n$, choose $p_{n+1} \in G \cap F_{n , p_n}$
\end{itemize}
Note that $p_{n+1}$ will be an extension of $p_n$ as otherwise $p_{n+1}$ and $ p_n$ have no common extensions (by the definition of $F_{n , p_n}$), but this cannot happen since $G$ is a filter.

For each $n$, suppose $p_n = (\alpha_n , r_n , \delta_n)$.  It follows that there exists $x \in \ell^2$ with
\[
x \in \bigcap_n \overline U_{r_n , \delta_n} = \bigcup_n U_{r_n , \delta_n} \subseteq \bigcap_n D_n = C_a
\]
where the ``$\subseteq$'' is a consequence of the fact that $p_{n+1} \in F_{n , p_n}$ for each $n$.  On the other hand, if $k \geq | \alpha_0 |$ with $n$ such that $k \in [ | \alpha_{n-1} , | \alpha_n | )$ (such an $n$ exists since $|\alpha_n| \geq n$).  By the choice of $p_n$, it follows that
\[
B^k_{\alpha_n^+} (x) \in \left\{ z \in \ell^2 : \lVert z \rVert_2 < 1\right\}.
\]
and hence
\[
B^k_{w_a} (x) \in \left\{ z \in \ell^2 : \lVert z \rVert_2 < 1\right\}
\]
since $\alpha_n$ is an initial segment of $w_a$.  In particular, $x \in C_a \setminus \HC (w_a)$ which proves the claim.

\noindent {\em Claim.}  For each $b < a$, $y_b \in \HC (w_a)$.

To establish this claim, fix a basic open set $U_{q , \varepsilon}$ and let $\alpha \subseteq w_a$ and $r, \delta$ be such that
\[
(\alpha , r , \delta ) \in G \cap E_{b , q , \varepsilon}.
\]
Hence, by the definition of $E_{b , q , \varepsilon}$, there exists $k \leq | \alpha |$ such that $B^k_{\alpha^+}$ maps $y_b$ nicely into $U_{q , \varepsilon}$.  Since $\alpha$ is an initial segment of $w_a$ and $k \leq |\alpha |$, it follows that $B^k_w$ also maps $y_b$ nicely into $U_{q , \varepsilon}$.  As $q, \varepsilon$ were arbitrary, this proves the claim.

To complete the construction, choose $y_a$ to be any element of
\[
\bigg( \bigcap_{b \leq a} \HC (w_b) \bigg) \cap C_a \cap {\sf NH}.
\]
Note that this set is nonempty by Martin's Axiom as it is an $|a|$-size intersection of comeager sets.
\end{proof}

\begin{remark}
Since all relevant partial orders in the proof above were countable, it would have been sufficient to make the weaker assumption that any $<\mathfrak c$-size intersection of comeager sets is nonempty.  This statement, known as ``${\sf non} (\mathcal M) < \mathfrak c$'' is an immediate consequence of Martin's Axiom.
\end{remark}


\begin{thebibliography}{99}

\bibitem{abakumov-gordon}
E. Abakumov and J. Gordon,
{\bf Common hypercyclic vectors for multiples of backward shift}, {\em Journal of functional analysis}, 2003

\bibitem{BayartMatheron}
F. Bayart, E. Matheron,
{\bf Dynamics of Linear Operators}, Cambridge University Press, 2009

\bibitem{jech}
T. Jech,
{\bf Set Theory}, Springer-Verlag 2003

\bibitem{kechris dst}
A. Kechris,
{\bf Classical Descriptive Set Theory}, Springer-Verlag 1995

\bibitem{rolewicz}
S.  Rolewicz,
{\bf On orbits of elements}, {\em Studia Mathematica}, 1969



\end{thebibliography}
\end{document}